\documentclass[10pt,preprint]{amsart}
\usepackage{amssymb,amsfonts,amsthm,amsmath,amscd,relsize}

\theoremstyle{theorem}
\newtheorem*{theorem}{Theorem}
\newtheorem*{corollary}{Corollary}
 
\theoremstyle{definition}

\usepackage{tikz}
\def\N{\mathbb N}
\def\Z{\mathbb Z}

\begin{document}

\title{Following in Yiu's Footsteps but on the Eisenstein Lattice}
\markright{Following in Yiu's Footsteps}

\author{Christian Aebi and Grant Cairns}

\address{Coll\`ege Calvin, Geneva, Switzerland 1211}
\email{christian.aebi@edu.ge.ch}
\address{Department of Mathematical and Physical Sciences, La Trobe University, Melbourne, Australia 3086}
\email{G.Cairns@latrobe.edu.au}

\maketitle

\begin{abstract}
Paul Yiu  proved that all Heron triangles are realizable on the integer lattice. 
We give an analogous result for triangles with vertices on the Eisenstein lattice.
\end{abstract}

%\section{Introduction} 
A planar triangle is called  a \emph{Heron triangle}, or a Heronian triangle, if it has integer side lengths and integer area. They are named after Hero of Alexandria 
($\sim$10--75AD), and it would be apt, and much nicer, to call them Hero triangles, but the terms Heron and Heronian are perhaps  too well established to change (?). Hero himself is sometimes called Heron, probably in the same way that Plato is called Platon in French  and in many other languages. 

Paul Yiu   \cite{Yiu} proved in this \emph{Monthly}  that all Heron triangles are realizable as triangles with vertices on the integer lattice. Considered in the complex plane, the integer lattice is the ring $\Z[i]$ generated by the elements $1$ and $i$, and the members of  $\Z[i]$ are called \emph{Gaussian integers}. 
In this note we give an analogous result for the \emph{Eisenstein lattice}, which is the lattice in the complex plane generated by the elements $1$ and $\omega=e^{2\pi i/ 3}=\frac12(-1+i\sqrt3)$. The elements of $\Z[\omega]$ are called \emph{Eisenstein integers}, named after the German mathematician Gotthold  Eisenstein (1823--1852).
%Note that $\omega^2+\omega+1=0$.
As usual, we denote the complex conjugate of $z\in  \Z[\omega]$ by $z^*$. Note that $\omega^*=\omega^2=-1-\omega$.
For $x,y\in\N$, the \emph{norm} of the element $z=x+y\omega\in \Z[\omega]$ is given by $N( z):=zz^*=x^2-xy+y^2$. 
The key property of the Eisenstein integers we will use below is that, like the Gaussian integers, they enjoy unique factorization (see \cite[Chapter~7]{St} and 
\cite[Chapter~9.1]{IR}); that is, every element can be written as a product of primes, and the factors are unique up to multiplication by a unit. The units in $\Z[\omega]$ are the six elements of norm one: $\{\pm1,\pm\omega,\pm(1+\omega)\}$; see Figure~\ref{F}. Since we will be discussing primes in $\Z$ and in $\Z[\omega]$, we will refer to the latter as  \emph{$\Z[\omega]$-primes}, and also use \emph{$\Z[\omega]$-prime} as an adjective. A key fact is that a prime $p\in\Z$ is $\Z[\omega]$-prime if and only if $p\equiv 2\pmod 3$ \cite[Chapter~9.1]{IR}.

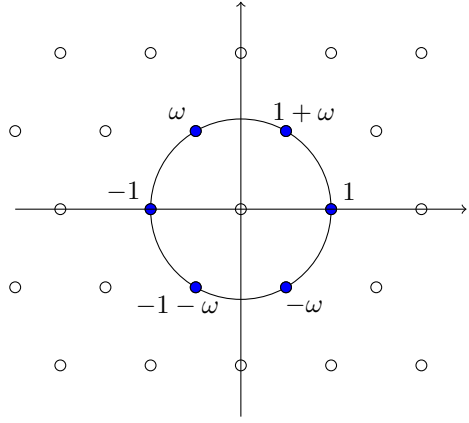
\begin{figure}[h]
\begin{center}
\begin{tikzpicture}[scale=1.2]
\def\r{1.732};

\foreach \i in {-2,...,2}
\foreach \j in {-1,...,1}
\draw (\i,\j*\r) circle (.06);
\foreach \i in {-2,...,2}
\foreach \j in {-1,...,0}
\draw (\i-1/2,\r*\j+\r/2) circle (.06);
%\draw (\i-1/2,\r*(2*\j+1)/2) circle (.06);

\draw[color=black] (-.7,\r/2+.2) node {$\omega$};
\draw[color=black] (1.2,.2) node {$1$};
\draw[color=black] (-1.3,.2) node {$-1$};
\draw[color=black] (-.7,-\r/2-.2) node {$-1-\omega$};
\draw[color=black] (.7,\r/2+.2) node {$1+\omega$};
\draw[color=black] (.7,-\r/2-.2) node {$-\omega$};

\draw[color=black]  (0,0) circle (1);

\draw[fill=blue]  (1,0) circle (.06);
\draw[fill=blue]  (-1,0) circle (.06);
\draw[fill=blue]  (.5,\r/2) circle (.06);
\draw[fill=blue]  (-.5,\r/2) circle (.06);
\draw[fill=blue]  (.5,-\r/2) circle (.06);
\draw[fill=blue]  (-.5,-\r/2) circle (.06);

 \draw [->] (0,-2.3) -- (0,2.3);
 \draw [->] (-2.5,0) -- (2.5,0);

  \end{tikzpicture}
\end{center}
\caption{The units in the ring of Eisenstein integers}\label{F}
\end{figure}

If a triangle has its vertices on the Eisenstein lattice, then its area is of the form $\frac{\sqrt3}4 n$, where $n\in\N$. Indeed,
for the triangle with vertices $0,x+y\omega,z+w\omega$, with $x,y,z,w\in\N$, the area is $\frac{\sqrt3}4 (xw-yz)$. The distance from the point $x+y\omega\in \Z[w]$ to the origin is $\sqrt{x^2-xy+y^2}$, and it is well known that for integers of the form $x^2-xy+y^2$, all prime divisors congruent to $2 \pmod 3$  have even exponent. Indeed, this follows readily from the description  given above of primes that are  $\Z[\omega]$-prime. Alternatively, one can use the fact that an integer can be written in the form $x^2-xy+y^2$ if and only if it can be written in the form $3m^2+n^2$; just consider 
\[
3\left(\frac{x}2\right)^2+\left(\frac{x}2-y\right)^2,\quad3\left(\frac{y}2\right)^2+\left(\frac{y}2-x\right)^2,\quad
3\left(\frac{x-y}2\right)^2+\left(\frac{x+y}2\right)^2,
\]
according to whether $x,y$ or $x+y$ is even \cite[p.~223]{CG}. The fact that for integers of the form $3m^2+n^2$, all prime divisors congruent to $2 \pmod 3$  have even exponent is a common number theory exercise; see \cite[p.~333]{AG}.  
% proof goes like this: use -3 is a quadratic residue mod p iff p =1 mod 3 (see eg https://math.stackexchange.com/questions/3776308/when-is-3-a-quadratic-residue-mod-p). 
% And use the fact that (a^2 + ab + b^2)(c^2 + cd + d^2) = e^2 + ef + f^2 where
% e = ac ? bd, f = ad + bd + bc
% and use Theorem 1.3, p. 104 of Rose, A course of number theory.
% see: https://guests.mpim-bonn.mpg.de/spc/teaching/primes_17/problems2_sol.pdf
% See Waldschmidt's notes: https://webusers.imj-prg.fr/~michel.waldschmidt/articles/pdf/NumbertTheoryWebSeminar14052020VI.pdf
%See also: https://web.northeastern.edu/dummit/teaching_sp21_4527/4527_lecture_29_factorization_diophantine_equations.pdf
% result proven in paper by Arlinghaus and Arlinghaus in journal on Geographical analysis (!?): https://onlinelibrary.wiley.com/doi/pdf/10.1111/j.1538-4632.1989.tb00882.x
% An elementary proof is also given in an unpublished paper by Nair: https://arxiv.org/pdf/math/0408107.pdf
For our purposes, the important consequence is that the length of the sides of a triangle with vertices on the Eisenstein lattice are of the form $r\sqrt{t} $, where $r,t\in\N$ and $t$ has no prime divisors congruent to $2 \pmod 3$.

\begin{theorem}\label{T:main}
A planar triangle $T$ with side lengths $a,b,c$ is realizable on the Eisenstein lattice if and only if the following three conditions hold:
\begin{enumerate}
\item[(i)]
the area of $T$ is of the form $\frac{\sqrt3}4 n$, where $n\in\N$, 
\item[(ii)] $a^2,b^2,c^2\in\N$,
\item[(iii)] one of the side lengths of $T$  is of the form $r\sqrt{t}$, where $r, t\in\N$  and  $t$ has no prime divisors congruent to $2 \pmod 3$.
\end{enumerate}
\end{theorem}

\begin{proof} 
Consider a planar triangle $T$ with side lengths $a,b,c$ with $a^2,b^2,c^2\in\N$ and area $\frac{\sqrt3}4 n$, where $n\in\N$. Suppose furthermore that $a=r\sqrt{t}$, where $r,t\in\N$  and $t$ has no prime divisors congruent to $2 \pmod 3$.
We will construct a triangle $ABC$, with vertices on the Eisenstein lattice, whose edge lengths agree with those of $T$.

Heron's formula 
 \cite[Chapter~6.7]{OW} for the area of $T$ gives
\[
\frac{\sqrt3}4 n=\sqrt{s(s-a )(s-b )(s-c )},
\]
where $s=\frac{(a +b +c )}2$ is the semi-perimeter.
Hence, 
\begin{equation}\label{E:heron0}
3n^2=(a +b +c )(a +b -c )(a -b +c )(-a +b +c ).
\end{equation}
Expanding and rearranging as a quadratic in $c^2$ gives
\begin{equation}\label{E:heron}
c^4-2(a^2+b^2)c^2 +(a^2-b^2)^2 +3n^2=0.
\end{equation}
As $c^2$ is an integer, the discriminant of \eqref{E:heron} is necessarily a square. Hence
\begin{equation}\label{E:c}
c^2=(a^2+b^2)+\Delta,
\end{equation}
where $\Delta\in\Z$ and 
  $\Delta^2= (a^2+b^2)^2-((a^2-b^2)^2 +3n^2)$. Thus
\begin{equation}\label{E:dis}
\Delta^2 +3n^2=4a^2b^2.
\end{equation}
From \eqref{E:dis}, the integers $\Delta$ and $n$ have the same parity, so we may set
\begin{equation}\label{E:uv}
\Delta=u+v,\qquad n=u-v,
\end{equation}
for $u,v\in \Z$.  Then $\Delta^2 +3n^2=4(u^2-uv+v^2)$ and so \eqref{E:dis} gives
\begin{equation}\label{E:dis2}
u^2-uv+v^2=a^2b^2.
\end{equation}
Let $z:=(1+\omega)(u+v\omega)$. Then $z$ has norm $zz^*=u^2-uv+v^2=a^2b^2$, by \eqref{E:dis2}. 
Factoring  in $\Z[\omega]$, it is convenient to write $z$ in the form $z=-f\cdot g$, where $N(f)=a^2$ and $N(g)=b^2$. The factors $f,g$ can be found as follows. Notice that because $t$ has no prime divisors congruent to $2 \pmod 3$, its prime divisors are not $\Z[\omega]$-prime; see \cite[Chapter~9.1]{IR}. Hence,
in the prime decomposition
of $a^2=r^2t$ in $\Z[\omega]$,  the $\Z[\omega]$-prime factors come in complex conjugate pairs. From each pair of factors, choose a factor which is a divisor of $z$, and let $f$ denote the product of these terms. So  $N(f)=a^2$. Then set $g=-z/f$, so  $N(g)=b^2$.

Now let $f=m+n\omega$ and $g=q+p\omega$, for $m,n,p,q\in \Z$. Notice that by construction, $m+n\omega$ has norm
\begin{equation}\label{E:mn}
m^2-mn+n^2=a^2,
\end{equation}
$q+p\omega$ has norm
\begin{equation}\label{E:pq}
p^2-pq+q^2=b^2,
\end{equation}
and $z=(1+\omega)(u+v\omega) =-(m+n\omega)(q+p\omega)$. 
Expanding both sides of the last equation gives
$u-v+u\omega=(-m q+ n p)+ (-m p- n q+ n p)\omega$,
and so the components $u,v$ are
\begin{equation}\label{E:3uv}
u= -m p+ n (p-q),\qquad
v=m (q-p)-nq.
\end{equation}
Thus, from \eqref{E:c}, \eqref{E:uv}, \eqref{E:mn}, \eqref{E:pq} and \eqref{E:3uv},
\begin{align}
c^2&=a^2+b^2+\Delta\notag\\
&=a^2+b^2+(u+v)\notag\\
&=(m^2-mn+n^2)+(p^2-pq+q^2)+m (q-2p)+n(p-2q)\notag\\
&=(m-p)^2-(m-p)(n-q)+(n-q)^2.\label{E:fin}
\end{align}
Consider the triangle $ABC$, where  $C$ is the origin, 
$B=m+n\omega$ and $A=p+q\omega$.
Then, as required, 
$BC$ has length $\sqrt{m^2-mn+n^2}=a$, by \eqref{E:mn},
$AC$ has length $\sqrt{p^2-pq+q^2}=b$, by \eqref{E:pq},
and
$AB$ has length
\[
\sqrt{(m-p)^2-(m-p)(n-q)+(n-q)^2}=c,
\]
 by \eqref{E:fin}.
\end{proof} 

Notice that the above theorem has the rather surprising corollary.

\begin{corollary} Suppose a planar triangle $T$ has area of the form $\frac{\sqrt3}4 n$, where $n\in\N$, and that the squares of the side lengths of $T$ are integers. If one of its side lengths is of the form $r\sqrt{t}$, where $r,t\in\N$  and $t$ has no prime divisors congruent to $2 \pmod 3$, then the other two sides have the same form.
\end{corollary}

%%%%%%%%%%%%%%%%%%%%%%%%%

%\begin{thks} The authors ....
%\end{thks}

%%%%%%%%%%%%%%%%%%%%%%%%%%%

\end{document}